\documentclass[12pt,a4paper]{amsart}
\usepackage[T1]{fontenc}
\usepackage[latin9]{inputenc}
\usepackage{amssymb}
\usepackage{color}
\usepackage{graphicx,color}
\usepackage{flafter}
\usepackage{amsmath, amssymb, graphics}
\makeatletter
\newtheorem{theorem}{Theorem}[section]

\theoremstyle{definition}
\newtheorem{definition}[theorem]{Definition}

\theoremstyle{remark}
\newtheorem{remark}[theorem]{Remark}

\numberwithin{equation}{section}
\def\R{\mathbb{R}}
\numberwithin{equation}{section} %% Comment out for sequentially-numbered
\numberwithin{figure}{section} %% Comment out for sequentially-numbered
  \@ifundefined{theoremstyle}{\usepackage{amsthm}}{}
  \theoremstyle{plain}
  \theoremstyle{remark}
  \theoremstyle{definition}
\providecommand{\keywords}[1]
{
\small	
\textbf{\textit{Keywords. }} #1
}
\begin{document}

\title[A new family of curves in space]{A new family of curves in space}
\author[H\'ector Efr\'en Guerrero Mora]
{H\'ector Efr\'en Guerrero Mora}
\date{\today}

\address{Departamento de Matem\'aticas.
	Universidad del Cauca \\ Facultad de Ciencias Naturales Exactas y de la Educaci\'on. Popay\'an, Colombia}
\email{heguerrero@unicauca.edu.co}
\thanks{The author was supported in part by Universidad del Cauca project ID 5124.}

\begin{abstract}

This article defines a new family of curves in space, whose graphs generate shapes similar to whirls.
An intrinsic equation is found, in terms of curvature and torsion, which gives necessary and sufficient conditions for the existence of this family. Its position vector is found, with arc length parameter and finally a way of generating a great variety of examples is shown; in particular an application to rectifying curves is given.
\end{abstract}
%%%%%%%%%%%%%%%%%%%%%%%%%
%    General info
%\date{January 1, 2001 and, in revised form, June 22, 2001.}
\dedicatory{To my family.}
\maketitle

%%%%%%%%%%%%%%%%%%%%%%5
\section{MOTIVATION.}
There is a great variety of examples that we can see in nature with swirl shapes; it is very common to find structures, in our environment, with these types of configurations. In the design of creation you can see from huge galaxies to tiny structures with these peculiar shapes.\\ The curves with whirl shapes, that is, curves with spiral type configurations, have captivated the minds of many geometricians. For this reason there are mathematical models that describe spiral configurations.
\section{WHIRL CURVE.}
In this section we will introduce one new families of curves and it is intended that with it, mathematical models can be generated that make an approach to the understanding of the formation and evolution of these swirling structures that appear in nature.
\newpage
\begin{definition}
	A whirl curve is a curve with curvature greater than zero, nonzero torsion, and presenting the property that the inner product formed by its normal vector with a fixed direction is proportional to the inner product formed by its tangent vector with that same direction.
\end{definition}
Next, we are going to show some properties of a whirl curve.
\begin{theorem}\label{Caracterizacion}
	Let $\alpha$ be a curve, with positive curvature $\kappa$ and nonzero torsion $\tau$, both differentiable functions.
	\\ $\alpha$ is a whirl curve if and only if its curvature $\kappa=\kappa(s)$ and its torsion $\tau=\tau(s)$ satisfy:
	\begin{equation}\label{ecuación intrinseca}
		\tau=\frac{(1+\lambda^2)(\frac{\tau}{\kappa})'}{\lambda(1+\lambda^2+(\frac{\tau}{\kappa})^2)}\ ,
	\end{equation}where $\lambda$ is the proportionality constant, with $\lambda\neq 0$.
\end{theorem}
\begin{proof}
	let $\alpha$ be a whirl curve and let $\textbf{t},\textbf{n}$ and $\textbf{b}$ be its tangent, normal, and binormal vectors, respectively. Then, there exists a constant vector $\textbf{d}$ of norm one such that
	\begin{equation*}
		\left <\textbf{n},\textbf{d}\right >=\lambda\left <\textbf{t},\textbf{d}\right >,
	\end{equation*} for some constant $\lambda$ other than zero.
	Deriving the previous expression and using the Frenet equations, we have
	\begin{equation*}
		\tau\left <\textbf{b},\textbf{d}\right >=(1+\lambda^2)\kappa\left <\textbf{t},\textbf{d}\right >.
	\end{equation*}
	\\Since $\{\textbf{t},\textbf{n},\textbf{b}\}$, forms an orthonormal basis of  $\R^3$, $\tau \neq 0$ and $\textbf{d}$ is a vector of norm one,  we can write
	\begin{eqnarray*}
		\textbf{d}&=&\left <\textbf{t},\textbf{d}\right >\textbf{t}+\lambda\left <\textbf{t},\textbf{d}\right >\textbf{n}+(1+\lambda^2)(\frac{\kappa}{\tau})\left <\textbf{t},\textbf{d}\right >\textbf{b},
	\end{eqnarray*}
	this is
	\begin{equation*}
		\left <\textbf{t},\textbf{d}\right >=\frac{\pm1}{\sqrt{(1+\lambda^2)^2(\frac{\kappa}{\tau})^2+(1+\lambda^2)}}.
	\end{equation*} This implies that
	\begin{equation*}
		\textbf{d}=\pm\frac{(\textbf{t}+\lambda\textbf{n}+(1+\lambda^2)(\frac{\kappa}{\tau})\textbf{b})}{\sqrt{(1+\lambda^2)^2(\frac{\kappa}{\tau})^2+(1+\lambda^2)}}.
	\end{equation*}
	And since  $\textbf{d}$ is a constant vector, without loss of generality, we can consider the positive sign and when finding its derivative, we have
	\begin{eqnarray*}
		\textbf{0}&=&-\frac{(1+\lambda^2)(\frac{\kappa}{\tau})(\frac{\kappa}{\tau})'+\lambda\kappa((1+\lambda^2)(\frac{\kappa}{\tau})^2+1)}{(1+\lambda^2)^{1/2}((1+\lambda^2)(\frac{\kappa}{\tau})^2+1)^{3/2}}\textbf{t}\\
		&-&\lambda\left (\frac{(1+\lambda^2)(\frac{\kappa}{\tau})(\frac{\kappa}{\tau})'+\lambda\kappa((1+\lambda^2)(\frac{\kappa}{\tau})^2+1)}{(1+\lambda^2)^{1/2}((1+\lambda^2)(\frac{\kappa}{\tau})^2+1)^{3/2}}\right)\textbf{n}\\
		&+&\frac{(1+\lambda^2)(\frac{\kappa}{\tau})(\frac{\kappa}{\tau})'+\lambda\kappa((1+\lambda^2)(\frac{\kappa}{\tau})^2+1)}{\frac{\kappa}{\tau}(1+\lambda^2)^{1/2}((1+\lambda^2)(\frac{\kappa}{\tau})^2+1)^{3/2}}\textbf{b}.
	\end{eqnarray*}And this implies that
	\begin{equation*}
		\tau=\frac{(1+\lambda^2)(\frac{\tau}{\kappa})'}{\lambda(1+\lambda^2+(\frac{\tau}{\kappa})^2)} \ ,
	\end{equation*}where $\lambda$ is the constant of proportionality other than zero.
	Now, suppose that the curve $\alpha=\alpha(s)$, with curvature $\kappa=\kappa(s)$ and torsion $\tau=\tau(s)$ satisfies the intrinsic equation (\ref{ecuación intrinseca})
	and let us define the next vector
	\begin{equation}\label{eje}
		\textbf{d}=\pm\frac{(\textbf{t}+\lambda\textbf{n}+(1+\lambda^2)(\frac{\kappa}{\tau})\textbf{b})}{\sqrt{(1+\lambda^2)^2(\frac{\kappa}{\tau})^2+(1+\lambda^2)}}.
	\end{equation}It is very easy to appreciate that vector $\textbf{d}$ is of norm one, that its derivative is equal to zero and a direct calculation shows that
	$\left <\textbf{n},\textbf{d}\right >-\lambda\left <\textbf{t},\textbf{d}\right >=0$, this is $\alpha$ is a whirl curve.
\end{proof}
\begin{definition}
	The vector $\textbf{d}$ defined by (\ref{eje}) is called the axis of the whirl curve. 
\end{definition}
The proof of the following theorem is based on a technique that appears in \cite{GueMora}.
\begin{theorem}\label{parametrizacion}
	Let $\alpha$ be a curve, with positive curvature $\kappa$ and nonzero torsion $\tau$, both continuously differentiable functions.
	\\$\alpha$ is a whirl curve if and only if its position vector $\alpha(s)=(x(s),y(s),z(s))$, except for a rigid motion, has natural representation of the form:
	\\
	\\
	$x(s)=\\
	\int\sqrt{1-\frac{e^{2(\lambda\int_{s_0}^s{\kappa}ds-B)}}{1+\lambda^2}}
	\cos(\arctan[\frac{\sqrt{1-e^{2(\lambda\int_{s_0}^s{\kappa}ds-B)}}}{\lambda}]-
	\frac{\emph{arctanh}[\sqrt{1-e^{2(\lambda\int_{s_0}^s{\kappa}ds-B)}}]}{\lambda})ds,$\\ \\
	$y(s)=\\
	\int\sqrt{1-\frac{e^{2(\lambda\int_{s_0}^s{\kappa}ds-B)}}{1+\lambda^2}}
	\sin(\arctan[\frac{\sqrt{1-e^{2(\lambda\int_{s_0}^s{\kappa}ds-B)}}}{\lambda}]-
	\frac{\emph{arctanh}[\sqrt{1-e^{2(\lambda\int_{s_0}^s{\kappa}ds-B)}}]}{\lambda})ds,$\\ \\
	$z(s)=\pm\int{\frac{e^{(\lambda\int_{s_0}^s{\kappa}ds-B)}}{\sqrt{1+\lambda^2}}}ds,$\ \  where $\lambda\int_{s_0}^s{\kappa}ds<B.$	\\ \\
\end{theorem}
\begin{proof}
	Suppose that $\alpha$ is a whirl curve.
	%%%%%%%%%%%%%%%%%%%%%%%%%%%%%%%%%%%%%%%%%%%%%%%%%%%%%%%%%%%%%%%%%%%%%%%%%%%
	Let us find a parametrization by arc length of the curve $\alpha.$
	Let us write its tangent vector $\textbf{t}=\textbf{t}(s)$ in spherical coordinates,
	\begin{equation}\label{vector tangente}
		\textbf{t} =(\sin \phi\cos \theta,\sin \phi\sin \theta,\cos \phi).
	\end{equation}Therefore, its normal vector $\textbf{n}=\textbf{n}(s)$ and its binormal vector $\textbf{b}=\textbf{b}(s)$ are given by
	\begin{eqnarray*}\textbf{n}&=&
		(\frac{\phi'\cos \phi\cos \theta-\theta'\sin \phi\sin \theta}{\kappa},\frac{ \phi'\cos \phi\sin \theta +\theta'\sin\phi\cos \theta}{\kappa},\frac{-\phi'\sin \phi}{\kappa}),
		\\ \textbf{b}&=&(\frac{-\phi'\sin \theta}{\kappa}-\frac{\theta'\sin 2\phi\cos\theta}{2\kappa},\frac{-\phi'\cos \theta}{\kappa}-\frac{\theta'\sin 2\phi\sin \theta}{2\kappa},\frac{\theta'\sin^2 \phi}{\kappa}).\end{eqnarray*}
	It is known that its Frenet trihedron $\textbf{t},\textbf{n},\textbf{b}$ forms an orthonormal basis of $\R^3$ and satisfies
	\begin{eqnarray*}
		\frac{d\textbf{t}}{ds}&=&\kappa \textbf{n},\\
		\frac{d\textbf{n}}{ds}&=&-\kappa\textbf{t}+\tau\textbf{b},\\
		\frac{d\textbf{b}}{ds}&=&-\tau \textbf{n}.
	\end{eqnarray*}For a fixed unit vector $D$, we have
	\begin{equation*}
		\left <\frac{d\textbf{t}}{ds},D\right >=\left<\kappa \textbf{n},D\right >=\kappa\left<\textbf{n},D\right>,
	\end{equation*} and 
	\begin{eqnarray*}
		\left<\frac{d\textbf{n}}{ds},D\right>&=&\left<-\kappa\textbf{t}+\tau\textbf{b},D\right>=
		-\kappa\left<\textbf{t},D\right>+\tau\left<\textbf{b},D\right>\\&=&-\kappa \left<\textbf{t},D\right>\pm\tau\sqrt{1-\left<\textbf{t},D\right>^2-\left<\textbf{n},D\right>^2}.
	\end{eqnarray*}
	Now, considering the definition of the curve $\alpha$, we have\\ $\left<\textbf{n},\textbf{d}\right>=\lambda\left< \textbf{t},\textbf{d}\right>$,
	for some constant vector $\textbf{d}$ of magnitude one. It is clear that there is an orthogonal transformation $\sigma:\R^3\rightarrow \R^3$, with a positive determinant such that $\sigma(\textbf{d})=(0,0,1)$ and knowing that the curvature $\kappa$ and the torsion $\tau$ are invariant, given a rigid motion, we have that the $\alpha$ and $\sigma\circ\alpha$ curves have the same curvature and torsion functions. We can assume that $D=\textbf{d}=(0,0,1)$. 
	Our goal is to find $\phi$ and $\theta.$
	%%%%%%%%%%%%%%%%%%%%%%%%%%%%%%%%%%%%%%%%%%%%%%%%%%%%%%%%%%%%%%%%%%%%%%%%%%%%
	Writing $\xi=\left< \textbf{t},\textbf{d}\right>$, let us look for a solution of the  form $\frac{1}{\kappa}\frac{d\xi}{ds}=\lambda\xi$,
	with $\xi\neq 0$, this is $\xi =\pm e^{(\lambda\int_{s_0}^s{\kappa ds}-C)}$,
	where $C$ is an integration constant.\\
	Now, let us use the intrinsic equation of curve $\alpha$, which is the same for curve $\sigma\circ \alpha.$
	\begin{equation*}
		\tau=\frac{(1+\lambda^2)(\frac{\tau}{\kappa})'}{\lambda(1+\lambda^2+(\frac{\tau}{\kappa})^2)} \ ,
	\end{equation*}where $\lambda$ is the proportionality constant. Which we can write it as:
	\begin{equation}\label{Curvatura}
		\kappa=\frac{(1+\lambda^2)(\frac{\tau}{\kappa})'}{\lambda(\frac{\tau}{\kappa})(1+\lambda^2+(\frac{\tau}{\kappa})^2)} \ ,
	\end{equation}integrating the above equation, we have:
	\begin{equation*}
		\int_{s_0}^s{\kappa(s)}ds=\frac{1}{\lambda}(\ln[\frac{1}{\sqrt{1+\lambda^2}}\mid\frac{\tau(s)}{\kappa(s)} \mid]-\ln[\sqrt{1+\frac{1}{1+\lambda^2}(\frac{\tau(s)}{\kappa(s)})^2}]+B),
	\end{equation*} where $B=-\displaystyle\ln(\frac{\frac{1}{\sqrt{1+\lambda^2}}\mid\frac{\tau(s_0)}{\kappa(s_0)}\mid}{\sqrt{1+\frac{1}{1+\lambda^2}(\frac{\tau(s_0)}{\kappa(s_0)})^2}})$.\\
	Then \begin{equation}\label{Exponencial}
		e^{\lambda\int_{s_0}^s\kappa ds-B}=\displaystyle\frac{\frac{1}{\sqrt{1+\lambda^2}}\mid\frac{\tau}{\kappa}\mid}{\sqrt{1+\frac{1}{1+\lambda^2}(\frac{\tau}{\kappa})^2}}<1,
	\end{equation} this implies
	\begin{equation*}
		\frac{1}{\sqrt{1+\lambda^2}}\frac{\tau}{\kappa}=\pm\frac{e^{(\lambda\int_{s_0}^s\kappa ds)}}{\sqrt{e^{2B}-e^{2\lambda\int_{s_0}^s\kappa ds}}},
	\end{equation*}
	this is
	\begin{equation*}
		\tau=\pm\frac{\sqrt{1+\lambda^2} \kappa e^{(\lambda\int_{s_0}^s\kappa ds-B)}}{\sqrt{1-e^{2(\lambda\int_{s_0}^s \kappa ds-B)}}},
	\end{equation*}
	Now, let us find the relationships between $C$ and $B$. To do this, consider the equation:
	\begin{eqnarray*}
		\frac{d}{ds}(\frac{1}{\kappa}\frac{d\xi}{ds})=\frac{d}{ds}\left<\frac{\textbf{t}'}{\kappa},\textbf{d}\right>&=&-\kappa \left<\textbf{t},\textbf{d}\right>\pm\tau\sqrt{1-\left<\textbf{t},\textbf{d}\right>^2-\left<\textbf{n},\textbf{d}\right>^2}\\
		&=&-\kappa \xi\pm\tau\sqrt{1-(1+\lambda^2)\xi^2},	
	\end{eqnarray*}
	and consider the case
	\begin{equation*}
		\xi =e^{(\lambda\int_{s_0}^s\kappa ds-C)},\ \	\tau=\frac{\sqrt{1+\lambda^2} \kappa e^{(\lambda\int_{s_0}^s \kappa ds-B)}}{\sqrt{1-e^{2(\lambda\int_{s_0}^s \kappa ds-B)}}}.
	\end{equation*}
	Replacing, we have
	\begin{eqnarray*}
		\lambda^2\kappa e^{(\lambda\int_{s_0}^s\kappa ds-C)}=&&-\kappa e^{(\lambda\int_{s_0}^s\kappa ds-C)}\\
		&&\pm\frac{\sqrt{1+\lambda^2} \kappa e^{(\lambda\int_{s_0}^s \kappa ds-B)}}{\sqrt{1-e^{2(\lambda\int_{s_0}^s \kappa ds-B)}}}\sqrt{1-(1+\lambda^2)e^{2(\lambda\int_{s_0}^s \kappa ds-C)}},
	\end{eqnarray*} this is
	\begin{equation*}
		\frac{(1+\lambda^2)}{\sqrt{1+\lambda^2}}e^{B-C}=\pm\frac{\sqrt{1-(1+\lambda^2)e^{2(\lambda\int_{s_0}^s \kappa ds-C)}}	}{\sqrt{1-e^{2(\lambda\int_{s_0}^s \kappa ds-B)}}},
	\end{equation*}and this implies that $(1+\lambda^2)e^{-2C}=e^{-2B},$
	the other cases not affecting the relationship between the constants $B$ and $C$.
	Therefore
	\begin{equation*}
		\cos\phi=\left<\textbf{t},(0,0,1)\right>=\left<\textbf{t},\textbf{d}\right>=\xi=\pm\frac{e^{(\lambda\int_{s_0}^s{\kappa ds}-B)}}{\sqrt{1+\lambda^2}}.
	\end{equation*}
	Now, to find the angle $\theta$, let us use the equation:
	\begin{equation*}
		\frac{\theta'\sin^2 \phi}{\kappa}=\left<\textbf{b},(0,0,1)\right>=\frac{1}{\tau}\left<\frac{d\textbf{n}}{ds},(0,0,1)\right>+\frac{\kappa}{\tau}\left<\textbf{t},(0,0,1)\right>.
	\end{equation*}
	Replacing
	\begin{equation*}
		\xi=\frac{e^{(\lambda\int_{s_0}^s{\kappa ds}-B)}}{\sqrt{1+\lambda^2}},\ \  \ \  \tau=\pm\frac{\sqrt{1+\lambda^2} \kappa e^{(\lambda\int_{s_0}^s \kappa ds-B)}}{\sqrt{1-e^{2(\lambda\int_{s_0}^s \kappa ds-B)}}},	
	\end{equation*} we have
	\begin{equation*}
		\theta'=\pm\frac{(1+\lambda^2)\kappa\sqrt{1-e^{2(\lambda\int_{s_0}^s \kappa ds-B)}}}{1+\lambda^2-e^{2(\lambda\int_{s_0}^s\kappa ds-B)}}, \ \ \text{respectively}.
	\end{equation*}
	Analogously, when considering the case
	\begin{equation*}
		\xi=-\frac{e^{(\lambda\int_{s_0}^s{\kappa ds}-B)}}{\sqrt{1+\lambda^2}},\ \  \ \  \tau=\pm\frac{\sqrt{1+\lambda^2} \kappa e^{(\lambda\int_{s_0}^s \kappa ds-B)}}{\sqrt{1-e^{2(\lambda\int_{s_0}^s \kappa ds-B)}}},	
	\end{equation*}
	we have
	\begin{equation*}
		\theta'=\mp\frac{(1+\lambda^2)\kappa\sqrt{1-e^{2(\lambda\int_{s_0}^s \kappa ds-B)}}}{1+\lambda^2-e^{2(\lambda\int_{s_0}^s\kappa ds-B)}}, \ \ \text{respectively}.
	\end{equation*}
	If we replace the different cases that $\xi$ and $\theta$ take in (\ref{vector tangente}), then several expressions are generated for the tangent vector, which are simplified in the following two, after applying the respective rigid motion. 
	\begin{eqnarray*}
		\textbf{t} =(&&\sqrt{1-\frac{e^{2(\lambda\int_{s_0}^s{\kappa}ds-B)}}{1+\lambda^2}}\cos \int{\frac{(1+\lambda^2)\kappa\sqrt{1-e^{2(\lambda\int_{s_0}^s \kappa ds-B)}}}{1+\lambda^2-e^{2(\lambda\int_{s_0}^s \kappa ds-B)}}}ds,\\ && \sqrt{1-\frac{e^{2(\lambda\int_{s_0}^s{\kappa}ds-B)}}{1+\lambda^2}}\sin \int{\frac{(1+\lambda^2)\kappa\sqrt{1-e^{2(\lambda\int_{s_0}^s \kappa ds-B)}}}{1+\lambda^2-e^{2(\lambda\int_{s_0}^s \kappa ds-B)}}}ds,\\&&\pm\frac{e^{(\lambda\int_{s_0}^s{\kappa}ds-B)}}{\sqrt{1+\lambda^2}}). 	
	\end{eqnarray*} And since 
	\begin{eqnarray*}
		&&\int{\frac{(1+\lambda^2)\kappa\sqrt{1-e^{2(\lambda\int_{s_0}^s \kappa ds-B)}}}{1+\lambda^2-e^{2(\lambda\int_{s_0}^s \kappa ds-B)}}}ds\\
		&=&\int{\frac{-\lambda^2\kappa e^{2(\lambda\int_{s_0}^s \kappa ds-B)}}{\sqrt{1-e^{2(\lambda\int_{s_0}^s \kappa ds-B)}}(1+\lambda^2-e^{2(\lambda\int_{s_0}^s \kappa ds-B)})}}ds\\&&+\int{\frac{\kappa}{\sqrt{1-e^{2(\lambda\int_{s_0}^s \kappa ds-B)}}}}ds\\&=&\arctan[\frac{\sqrt{1-e^{2(\lambda\int_{s_0}^s{\kappa}ds-B)}}}{\lambda}]-
		\frac{\emph{arctanh}[\sqrt{1-e^{2(\lambda\int_{s_0}^s{\kappa}ds-B)}}]}{\lambda},	
	\end{eqnarray*} with the integration constant equal to zero, the desired implication is obtained, this means we have the representation of the position vector, except for a rigid motion.\\
	To prove the reciprocal, it is enough to check that the curvature and torsion of the natural representation of the curve $\alpha$ are given by
		\begin{equation*}
		\kappa_{\alpha}=\mid\mid \alpha''\mid\mid=\kappa,\ \ \ \	\tau_{\alpha}=\frac{\alpha'\wedge \alpha''\cdot \alpha'''}{\kappa^2}=\pm\frac{\sqrt{1+\lambda^2} \kappa e^{(\lambda\int_{s_0}^s \kappa ds-B)}}{\sqrt{1-e^{2(\lambda\int_{s_0}^s \kappa ds-B)}}}, \ \text{respectively}.
	\end{equation*}
	And a direct calculation shows that the curvature $\kappa_{\alpha}$ and the torsion $\tau_{\alpha}$ satisfy the intrinsic equation:
	\begin{equation*}
		\frac{(1+\lambda^2)(\frac{\tau_{\alpha}}{\kappa_{\alpha}})'}{\lambda(1+\lambda^2+(\frac{\tau_{\alpha}}{\kappa_{\alpha}})^2)}=\tau_{\alpha},\ \text{respectively}.
	\end{equation*}
\end{proof}
\begin{remark}\label{Ecuacion simplificada}
	The whirl curve $\alpha(s)=(x(s),y(s),z(s))$ can be written in the following two ways
	\begin{eqnarray*}
		x(s)&=&\int\frac{\lambda}{\sqrt{1+\lambda^2}}\cos(\frac{\emph{arctanh}[\sqrt{1-e^{2(\lambda\int_{s_0}^s{\kappa}ds-B)}}]}{\lambda})\\
		&&+\frac{\sqrt{1-e^{2(\lambda\int_{s_0}^s{\kappa}ds-B)}}}{\sqrt{1+\lambda^2}}\sin(\frac{\emph{arctanh}[\sqrt{1-e^{2(\lambda\int_{s_0}^s{\kappa}ds-B)}}]}{\lambda})ds,\\ \\
		y(s)&=&\int\frac{\sqrt{1-e^{2(\lambda\int_{s_0}^s{\kappa}ds-B)}}}{\sqrt{1+\lambda^2}}\cos(\frac{\emph{arctanh}[\sqrt{1-e^{2(\lambda\int_{s_0}^s{\kappa}ds-B)}}]}{\lambda})
		\\&&-\frac{\lambda}{\sqrt{1+\lambda^2}}\sin(\frac{\emph{arctanh}[\sqrt{1-e^{2(\lambda\int_{s_0}^s{\kappa}ds-B)}}]}{\lambda})ds,\\ \\
		z(s)&=&\pm\int{\frac{e^{(\lambda\int_{s_0}^s{\kappa}ds-B)}}{\sqrt{1+\lambda^2}}}ds, \ \text{where}\  \lambda\int_{s_0}^s{\kappa}ds<B,
	\end{eqnarray*}
	with curvature $\kappa_{\alpha}=\kappa$ and torsion $\tau_{\alpha}=\pm\frac{\sqrt{1+\lambda^2}  e^{(\lambda\int_{s_0}^s \kappa ds-B)}\kappa}{\sqrt{1-e^{2(\lambda\int_{s_0}^s\kappa ds-B)}}}$, respectively.\\ 
\end{remark}
\section{APPLICATIONS}
By writing the quotient $\frac{\tau(s)}{\kappa(s)}=h(s)$ and replacing in (\ref{Curvatura}) and (\ref{Exponencial}), we have
\begin{equation}\label{Curvatura y formula}
	\kappa(s)=\frac{(1+\lambda^2)h'(s)}{\lambda h(s)(1+\lambda^2+(h(s))^2)}\ \ \text{and} \ 	\ 	e^{\lambda\int_{s_0}^s\kappa(s) ds-B}=\displaystyle\frac{\frac{1}{\sqrt{1+\lambda^2}}\mid h(s)\mid}{\sqrt{1+\frac{1}{1+\lambda^2}(h(s))^2}}
\end{equation}  and these expressions allow us to construct a great variety of interesting curves.
\subsection{APPLICATIONS TO THE RECTIFYING CURVES}
The notion of rectifying curve has been introduced by Chen and is defined as a unit speed curve such that position vector always lies in its rectifying plane. He found a simple characterization in terms of the ratio $\tau/\kappa$, for the rectifying curves. He proved the following theorem (For the proof of the theorem, see \cite{ChenBY:03})
\begin{theorem}\label{Caracterizacion 2}
	Let $\textbf{x}:I\rightarrow \mathbb{R}^3$ be a curve with $\kappa>0$. Then $\textbf{x}$ is congruent to a rectifying curve if and only if the ratio of torsion and curvature of the curve is a nonconstant linear function in arclength function $s$, i.e.,$\tau/\kappa=c_1s+c_2$ for some constants $c_1$ and $c_2$, with $c_1\neq 0.$
\end{theorem}
In this section we will find a  whirl curve that satisfies the definition of being a rectifying curve. For this we consider the theorem \ref{Caracterizacion}, the theorem \ref{parametrizacion}, the remark \ref{Ecuacion simplificada}, the theorem   \ref{Caracterizacion 2} and the formulas (\ref{Curvatura y formula}). The proof is not difficult, it just takes a bit of work to calculate the integral and verify that the curvature and torsion satisfy the respective intrinsic equations.
\begin{theorem}
	Let $\sigma:I\rightarrow \mathbb{R}^3$ be a curve, with positive $\kappa_{\sigma}=\kappa_{\sigma}(s)$ curvature and nonzero $\tau_{\sigma}=\tau_{\sigma}(s)$ torsion, both differentiable functions.\\
	Then $\sigma$ is congruent to a whirl curve and to a rectifying curve if and only if its position vector $\sigma(s)=(x(s),y(s),z(s))$, except for a rigid motion, has natural representation of the form:
	\begin{eqnarray*}
		x(s)&=&\frac{(b+as)}{a}\frac{\lambda}{\sqrt{1+\lambda^2}}\cos{[\frac{arctanh(\sqrt{\frac{1+\lambda^2}{1+(b+as)^2+\lambda^2}})}{\lambda}]},\\
		y(s)&=&-\frac{(b+as)}{a}\frac{\lambda}{\sqrt{1+\lambda^2}}\sin{[\frac{arctanh(\sqrt{\frac{1+\lambda^2}{1+(b+as)^2+\lambda^2}})}{\lambda}]},\\
		z(s)&=&\frac{1}{a}\frac{\sqrt{1+(b+as)^2+\lambda^2}}{\sqrt{1+\lambda^2}},
	\end{eqnarray*} where $\lambda$, $a$ and $b$ are constants, with $\lambda\neq 0$, $a\neq 0$ and the curve $\sigma$ is defined for $as+b>0$, respectively $as+b<0$.
\end{theorem}
The previous results motivate us to give the following definition.
\begin{definition}
	Let $\sigma:I\rightarrow \R^3$ be a unit speed curve with curvature $\kappa_{\sigma}>0$ and let $\{\kappa_{\sigma},\tau_{\sigma},\textbf{t},\textbf{n},\textbf{b} \}$ be the Frenet-Serret apparatus of $\sigma$. Then $\sigma$ is said to be a whirl-rectifying curve if it satisfies the following two conditions:
	\begin{enumerate}
		\item	There exists a constant vector $\textbf{d}$ of norm one such that\\
		$\left<\textbf{n}(s),\textbf{d}\right>=\lambda\left<\textbf{t}(s),\textbf{d}\right>$, for some constant $\lambda$ other than zero, 
		\item $\left<\sigma(s),\textbf{n}(s)\right>=0.$
	\end{enumerate}
\end{definition}
\begin{theorem}
	Let $\sigma:I\rightarrow \mathbb{R}^3$ be a curve, with positive curvature $\kappa_{\sigma}=\kappa_{\sigma}(s)$ and nonzero torsion $\tau_{\sigma}=\tau_{\sigma}(s)$, both differentiable functions.\\
	If $\sigma$ is a whirl-rectifying curve given by $\sigma(s)=(x(s),y(s),z(s))$, where
	\begin{eqnarray*}
		x(s)&=&\frac{(b+as)}{a}\frac{\lambda}{\sqrt{1+\lambda^2}}\cos{[\frac{arctanh(\sqrt{\frac{1+\lambda^2}{1+(b+as)^2+\lambda^2}})}{\lambda}]},\\
		y(s)&=&-\frac{(b+as)}{a}\frac{\lambda}{\sqrt{1+\lambda^2}}\sin{[\frac{arctanh(\sqrt{\frac{1+\lambda^2}{1+(b+as)^2+\lambda^2}})}{\lambda}]},\\
		z(s)&=&\frac{1}{a}\frac{\sqrt{1+(b+as)^2+\lambda^2}}{\sqrt{1+\lambda^2}},
	\end{eqnarray*}where $\lambda$, $a$ and $b$ are constants, with $\lambda\neq 0$, $a\neq 0$ and the curve $\sigma$ is defined for\\ $as+b>0$, respectively $as+b<0$, then $\sigma$ lives in two-leaf hyperboloid
	\begin{equation*} z^2-\frac{x^2}{\lambda^2}-\frac{y^2}{\lambda^2}=\frac{1}{a^2},\end{equation*} and is also a geodesic curve on the cone with vertex at the origin and parametrized by $\textbf{X}(t,u)=u\textbf{w}(t)$, where $u\in \R^+$ and $w=w(t)$ is a curve with unit speed on the unitary sphere with center at the origin and parametrized by
	\\ \\
	$\textbf{w}(t)=(\frac{\mid a\mid}{a}\frac{\lambda}{\sqrt{1+\lambda^2}}\sin[d+t]\cos[\frac{arctanh[\frac{\sqrt{1+\lambda^2}}{\sqrt{1+\tan^2[d+t]+\lambda^2}}]}{\lambda}],\\ \\ -\frac{\mid a\mid}{a}\frac{\lambda}{\sqrt{1+\lambda^2}}\sin[d+t]\sin[\frac{arctanh[\frac{\sqrt{1+\lambda^2}}{\sqrt{1+\tan^2[d+t]+\lambda^2}}]}{\lambda}],\frac{\mid a\mid}{a}\cos[d+t]\frac{\sqrt{1+\tan^2[d+t]+\lambda^2}}{\sqrt{1+\lambda^2}}),
	$\\ \\
	defined for $-d<t<-d+\frac{\pi}{2}$, respectively $-d-\frac{\pi}{2}<t<-d.$
\end{theorem}
\begin{proof}
	The curve is in the two-leaf hyperboloid, since
	\begin{equation*}
		z(s)^2-\frac{(x(s)^2+y(s)^2)}{\lambda^2}=\frac{1+(b+as)^2+\lambda^2}{a^2(1+\lambda^2)}-\frac{(b+as)^2}{a^2(1+\lambda^2)}=\frac{1}{a^2}.
	\end{equation*}Now taking 
	\begin{eqnarray*}t(s)&=&-d+\arctan{(b+as)},\\ u(s)&=&\frac{\sqrt{1+(b+as)^2}}{\mid a\mid},
	\end{eqnarray*} we see that $\textbf{X}(t(s),u(s))=u(s)\textbf{w}(t(s))=\sigma(s)$, which shows that the curve is on the cone.
	A direct calculation shows that the normal vector of the cone at point $\textbf{X}(t(s),u(s))$ is parallel to the normal vector of the curve at point $\sigma(s)$,
	which shows that $\sigma$ is a geodesic curve \cite{BANG-Y-C:07}, \cite{GueMoraDos}.
\end{proof}
\section{CONTINUOUS EXTENSIONS IN WHIRL CURVES}
Considering the equations that appear in (\ref{Curvatura y formula}), and writing $h(s)=as+b$, we have
\begin{equation*}
	\kappa(s)=\frac{(1+\lambda^2)a}{\lambda (as+b)(1+\lambda^2+(as+b)^2)}\ \ \text{and} \ 	\ 	e^{\lambda\int_{s_0}^s\kappa(s) ds-B}=\displaystyle\frac{\frac{1}{\sqrt{1+\lambda^2}}\mid as+b\mid}{\sqrt{1+\frac{1}{1+\lambda^2}(as+b)^2}}.
\end{equation*}And it is clear that $\kappa=\kappa(s)$ is not defined in all real numbers.\\
The whirl-rectifying curve $\sigma$ can be extended continuously to the whole set of real numbers. Additionally, its respective radial projection $\textbf{w}=\textbf{w}(t)$, on the unitary sphere with center at the origin, can be extended continuously to the entire interval\\ $(-d-\frac{\pi}{2},-d+\frac{\pi}{2})$, as the following results show 
\begin{theorem}
	Let $a$, $b$ and $\lambda$ be constant with $a\neq 0$ and $\lambda\neq 0$.\\
	The curve $\varOmega:\R\rightarrow\R^3$	given by $\varOmega(s)=(x(s),y(s),z(s))$, where
	\begin{eqnarray*}
		x(s)&=&\frac{(b+as)}{a}\frac{\lambda}{\sqrt{1+\lambda^2}}\cos{[\frac{arctanh(\sqrt{\frac{1+\lambda^2}{1+(b+as)^2+\lambda^2}})}{\lambda}]},\\ 
		y(s)&=&-\frac{(b+as)}{a}\frac{\lambda}{\sqrt{1+\lambda^2}}\sin{[\frac{arctanh(\sqrt{\frac{1+\lambda^2}{1+(b+as)^2+\lambda^2}})}{\lambda}]},\\
		z(s)&=&\frac{1}{a}\frac{\sqrt{1+(b+as)^2+\lambda^2}}{\sqrt{1+\lambda^2}};\  \text{defined for}\ s\neq -\frac{b}{a}, \ \text{and}
	\end{eqnarray*} 
	\begin{eqnarray*}
		x(s)&=&0,\\ y(s)&=&0,\\ z(s)&=&\frac{1}{a};\  \text{defined for}\ s= -\frac{b}{a}.
	\end{eqnarray*}
	is continuous in all the set of the	real numbers and the restrictions $\varOmega\mid_{(-\infty,-\frac{b}{a})}$, $\varOmega\mid_{(-\frac{b}{a},\infty)}$ are whirl-rectifying curves.
\end{theorem}
\begin{proof}
	From the following inequalities: 
	\\ \\
	$
	\mid\frac{(b+as)}{a}\frac{\lambda}{\sqrt{1+\lambda^2}}\mid	\leq\frac{(b+as)}{a}\frac{\lambda}{\sqrt{1+\lambda^2}}\cos{[\frac{arctanh(\sqrt{\frac{1+\lambda^2}{1+(b+as)^2+\lambda^2}})}{\lambda}]}\leq -\mid\frac{(b+as)}{a}\frac{\lambda}{\sqrt{1+\lambda^2}}\mid,\\ \\
	\mid\frac{(b+as)}{a}\frac{\lambda}{\sqrt{1+\lambda^2}}\mid	\leq\frac{(b+as)}{a}\frac{\lambda}{\sqrt{1+\lambda^2}}\sin{[\frac{arctanh(\sqrt{\frac{1+\lambda^2}{1+(b+as)^2+\lambda^2}})}{\lambda}]}\leq -\mid\frac{(b+as)}{a}\frac{\lambda}{\sqrt{1+\lambda^2}}\mid,
	$ \\ \\
	and from the definition of whirl-rectifying curve, the result follows.
\end{proof}
In a similar way, the following theorem is proved.
\begin{theorem}
	Let $a$, $b$ and $\lambda$ be constant with $a\neq 0$ and $\lambda\neq 0$.\\
	The curve $\varUpsilon:\R\rightarrow\R^3$	given by $\varUpsilon(t)=(x(t),y(t),z(t))$, where
	\begin{eqnarray*}
		x(t)&=&\frac{\mid a\mid}{a}\frac{\lambda}{\sqrt{1+\lambda^2}}\sin[d+t]\cos[\frac{arctanh[\frac{\sqrt{1+\lambda^2}}{\sqrt{1+\tan^2[d+t]+\lambda^2}}]}{\lambda}],\\ y(t)&=&-\frac{\mid a\mid}{a}\frac{\lambda}{\sqrt{1+\lambda^2}}\sin[d+t]\sin[\frac{arctanh[\frac{\sqrt{1+\lambda^2}}{\sqrt{1+\tan^2[d+t]+\lambda^2}}]}{\lambda}],\\
		z(t)&=&\frac{\mid a\mid}{a}\cos[d+t]\frac{\sqrt{1+\tan^2[d+t]+\lambda^2}}{\sqrt{1+\lambda^2}}, 
	\end{eqnarray*} defined for $ -d-\frac{\pi}{2}<t<-d+\frac{\pi}{2},\ t\neq -d$,  and 
	\begin{eqnarray*}
		x(t)&=&0,\\ y(t)&=&0,\\ z(t)&=&\frac{\mid a\mid}{a};\  \text{defined for}\ t= -d,
	\end{eqnarray*}
	is continuous in $(-d-\frac{\pi}{2},-d+\frac{\pi}{2})$ and 
	\begin{eqnarray*}
		u\varUpsilon\mid_{(-d-\frac{\pi}{2},-d)}(t)&=&\varOmega\mid_{(-\infty,-\frac{b}{a})}(s),\\
		u\varUpsilon\mid_{(-d,-d+\frac{\pi}{2})}(t)&=&\varOmega\mid_{(-\frac{b}{a},\infty)}(s),
	\end{eqnarray*}where   
	\begin{eqnarray*}t&=&t(s)=-d+\arctan{(b+as)},\\u&=&u(s)=\frac{\sqrt{1+(b+as)^2}}{\mid a\mid}.
	\end{eqnarray*}
\end{theorem}
\subsection{EXAMPLES}
If we take $a=0.65, b=0$ we can graph $\varOmega=\varOmega(s)$ for the different values of $\lambda=-20,-4,-1.8,-1,-0.5,-0.26$. In an analogous way it is graphed for their respective radial projections $\varUpsilon=\varUpsilon(t)$, on the unitary sphere with center at the origin.
\begin{figure}[!ht]
	\begin{center}
		\begin{tabular}{cccc}
			\includegraphics[viewport=0 0 700 849,scale=0.4,clip,scale=0.4,clip]{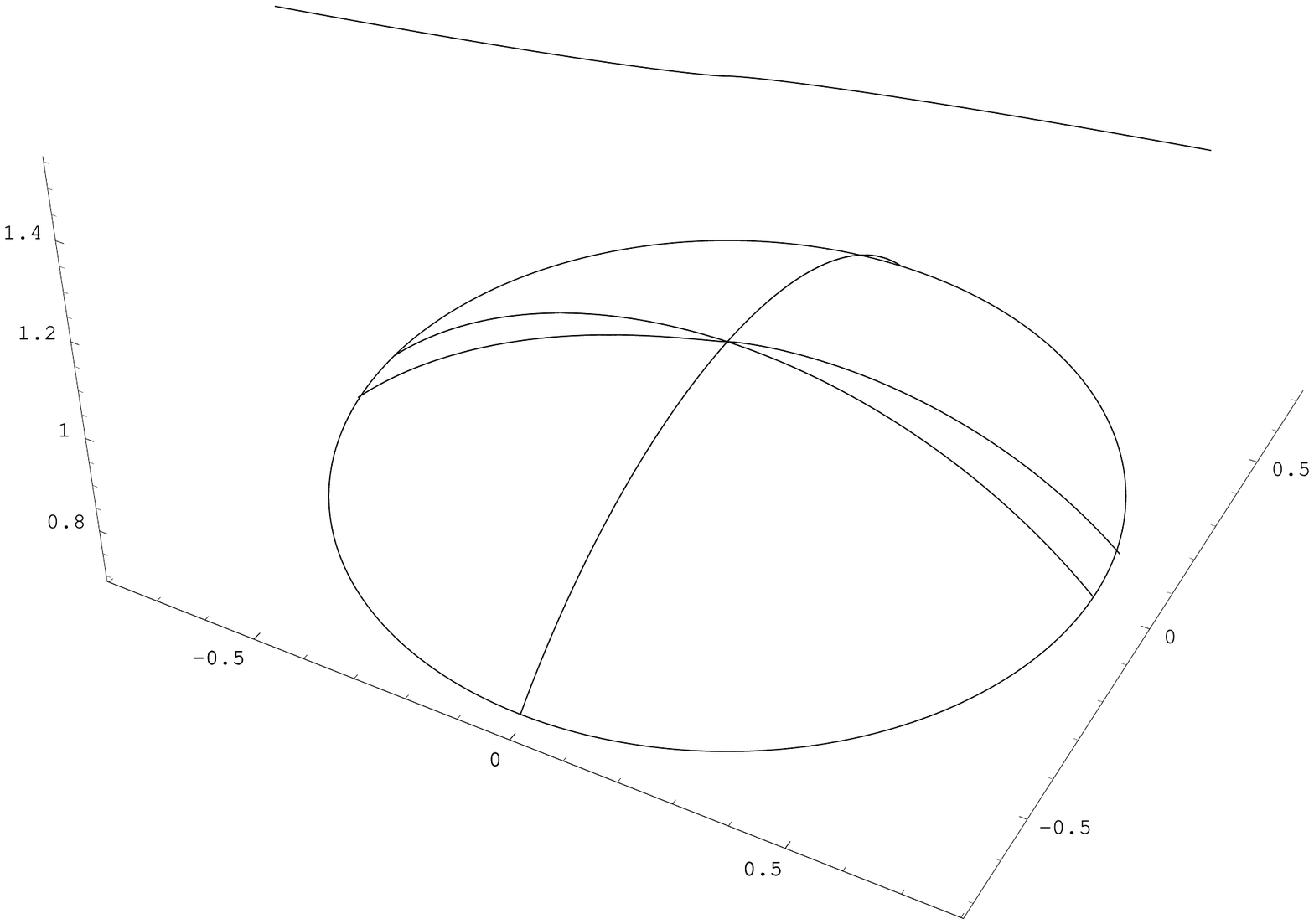} &\includegraphics[viewport=0 0 700 849,scale=0.4,clip,scale=0.4,clip]{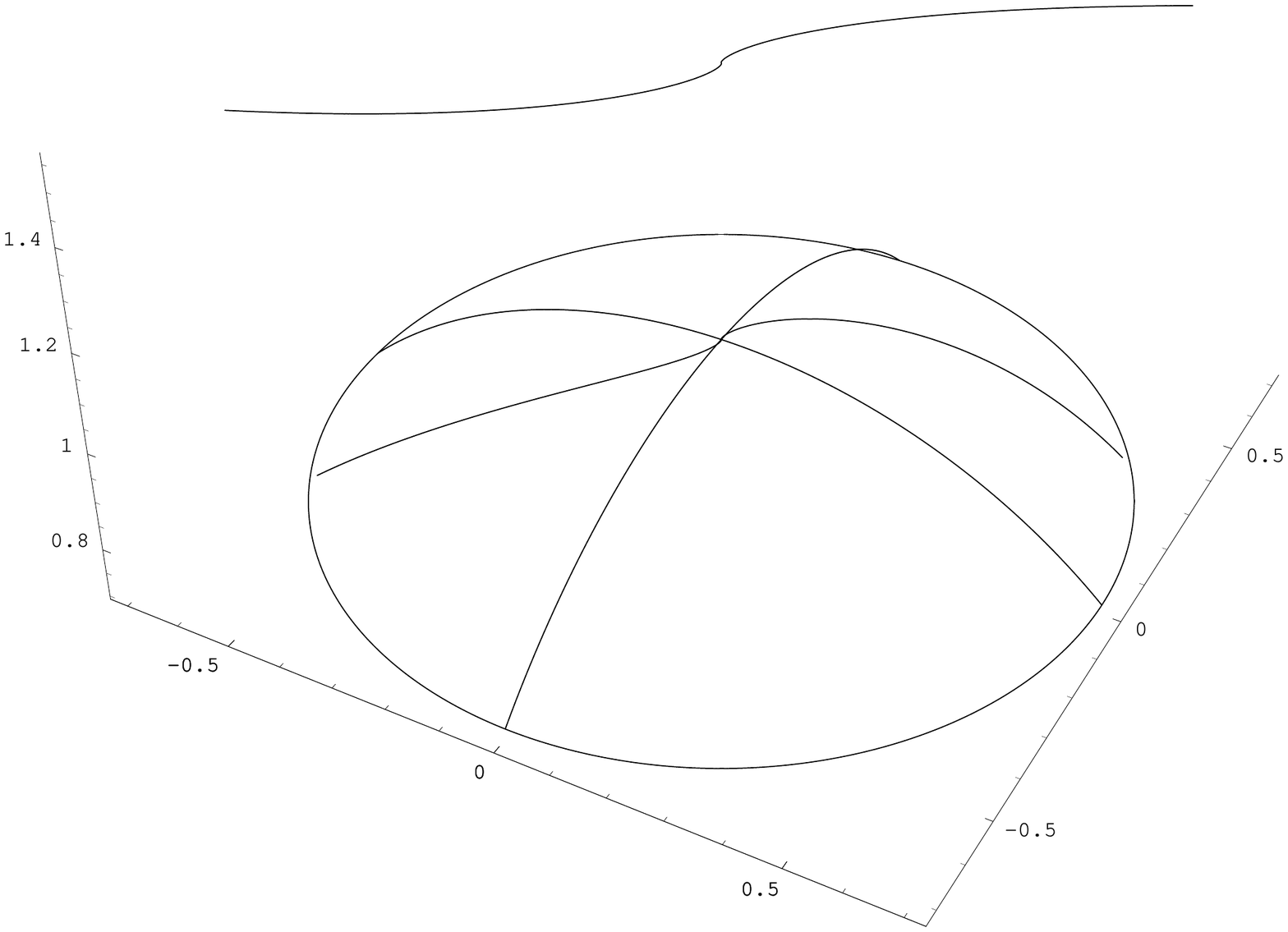} &\includegraphics[viewport=0 0 700 849,scale=0.4,clip,scale=0.4,clip]{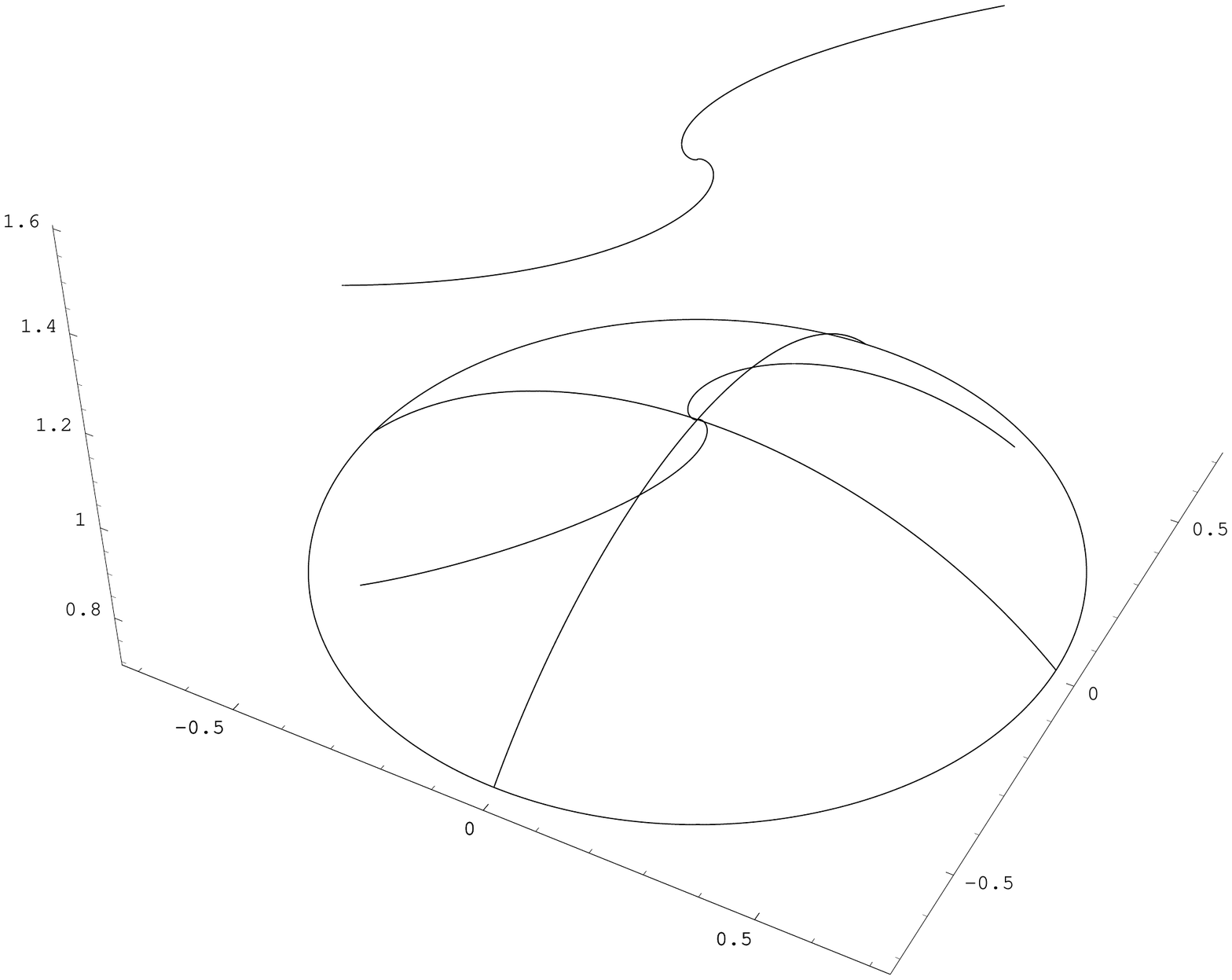} \\
			\text{$\lambda=-20$}&\text{$\lambda=-4$}&\text{$\lambda=-1.8$}&\\
			\includegraphics[viewport=0 0 700 849,scale=0.4,clip,scale=0.4,clip]{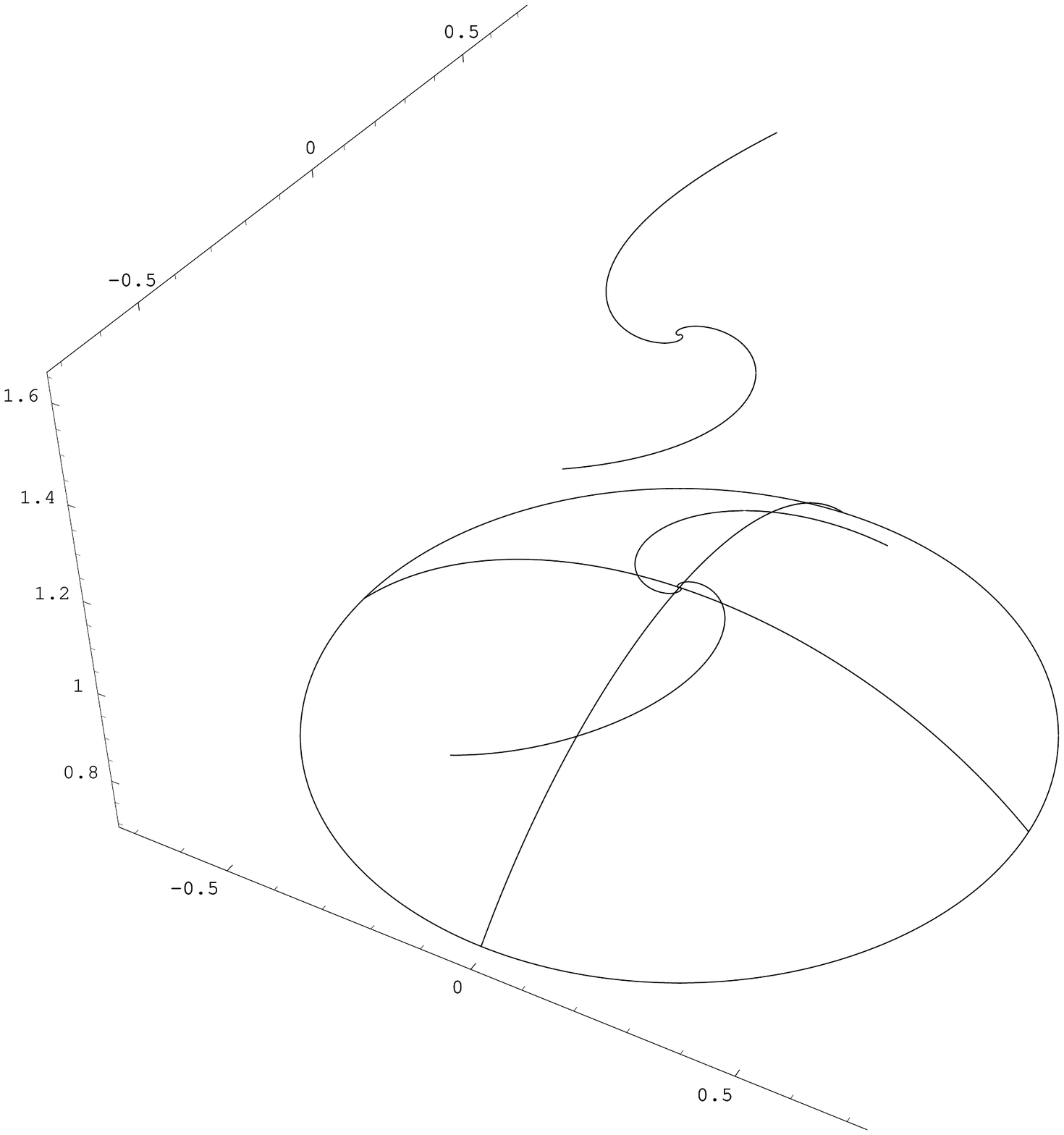} &\includegraphics[viewport=0 0 700 849,scale=0.4,clip,scale=0.4,clip]{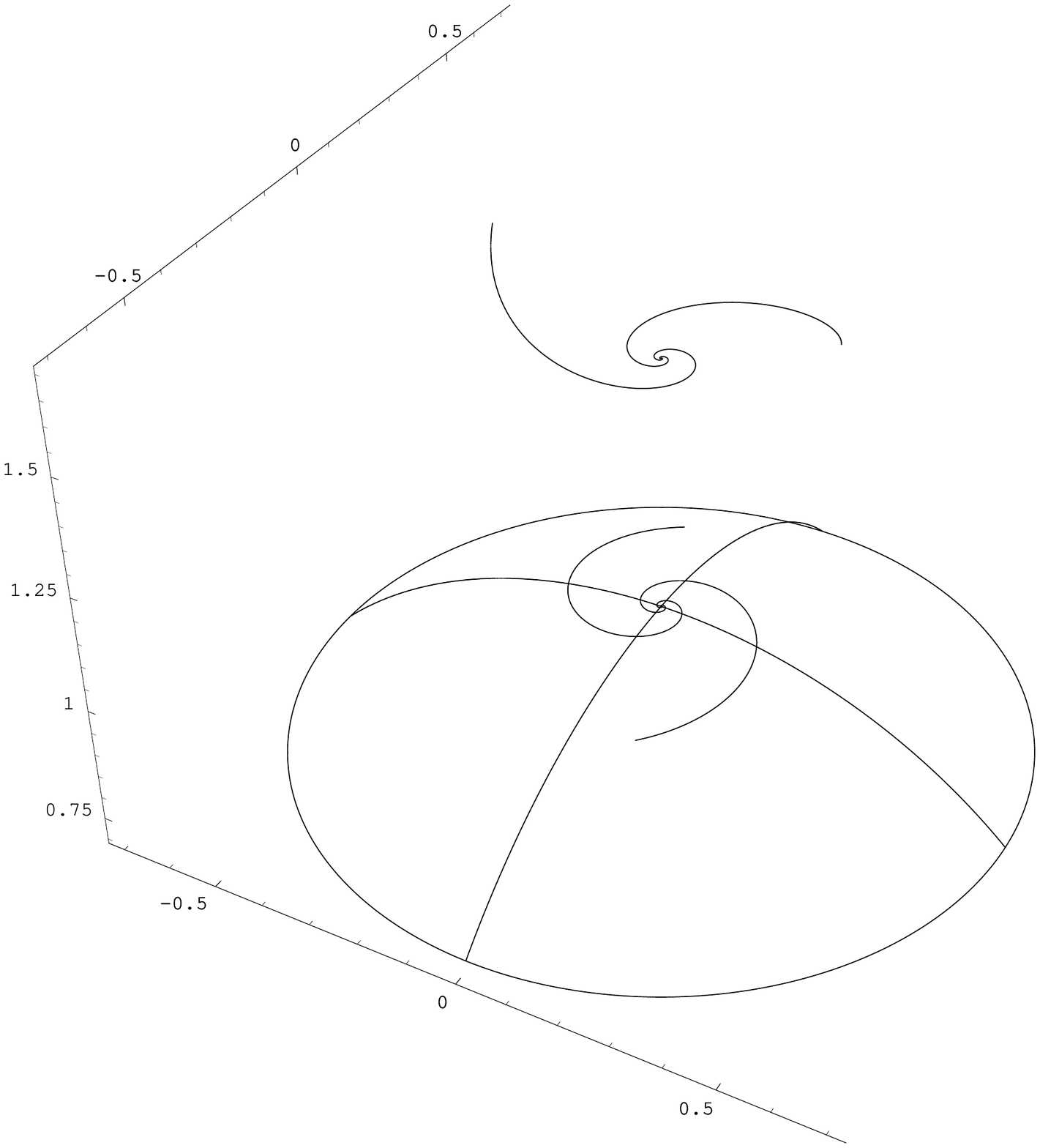} &\includegraphics[viewport=0 0 700 849,scale=0.4,clip,scale=0.4,clip]{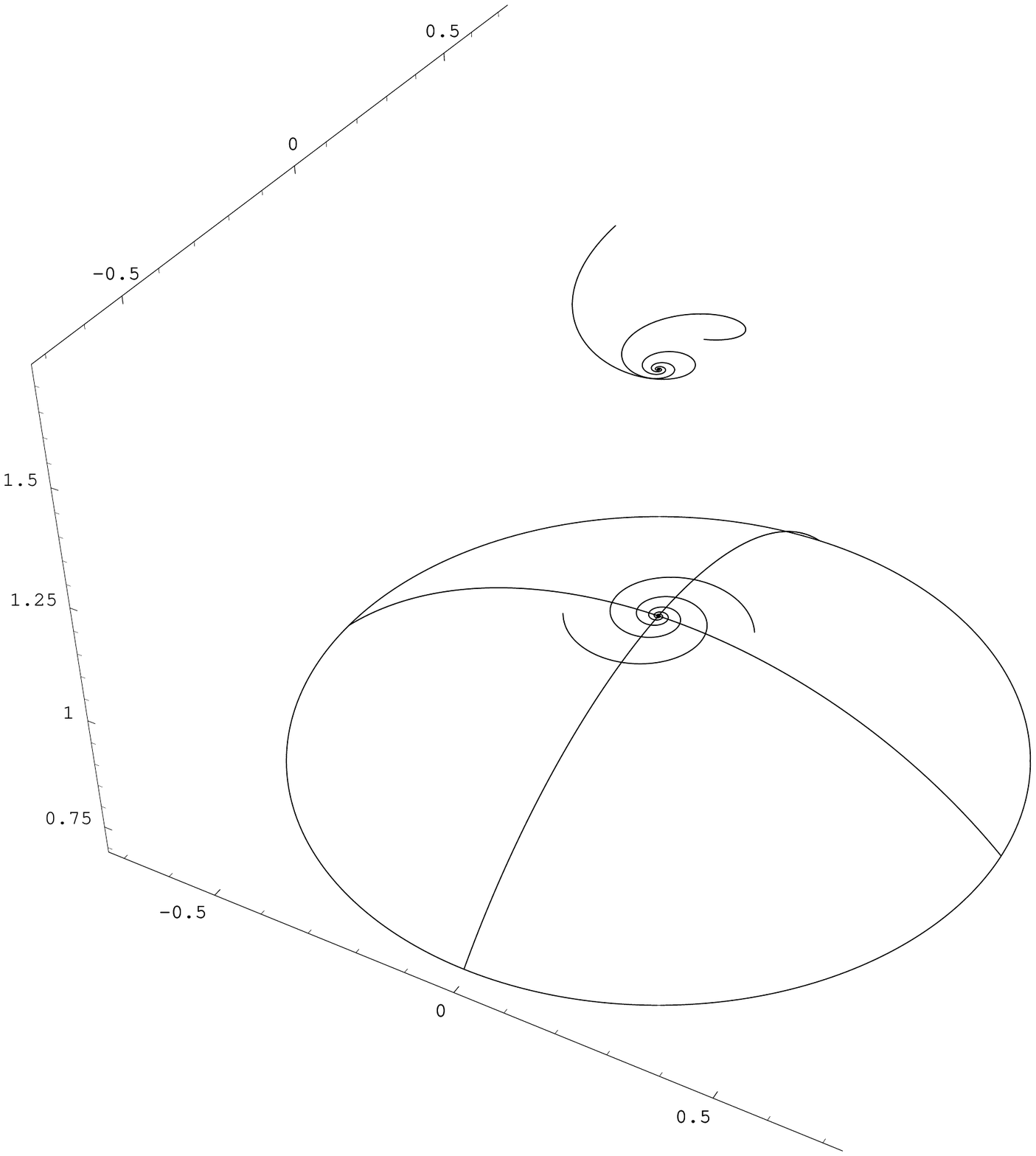} \\
			\text{$\lambda=-1$}&\text{$\lambda=-0.5$}&\text{$\lambda=-0.26$}&
		\end{tabular}
	\end{center}
	\caption{ Some $\varOmega=\varOmega(s)$ extensions of the whirl-rectifing curve with their respective radial projections $\varUpsilon=\varUpsilon(t)$, on the unitary sphere with center at the origin, for $a=0.65, b=d=0$, plotted for $s,t\in[-\frac{\pi}{4},\frac{\pi}{4}]$.}
	%\end{figure}
	%\begin{figure}[!ht]
\end{figure}
%%%%%%%%%%%%%%%%%%%%%%%%%%%%%%%%%%%%%%%%%%%%%%%%%%%%%%%%%%%%%%%%%%%%%%%%%%%
\bibliographystyle{amsplain}

\end{document}